\documentclass{article}
\usepackage{fullpage}

\usepackage[utf8]{inputenc}
\usepackage[T1]{fontenc}
\usepackage{amsmath}
\usepackage{amsthm,amssymb}
\usepackage{nicefrac,amsfonts}
\usepackage{thmtools,mathtools}
\usepackage{caption,subcaption}
\usepackage{graphicx,color,tikz,tcolorbox}
\usepackage{enumerate,enumitem}
\usepackage[sort,nocompress]{cite}
\usepackage{xspace}
\usepackage{bbm}
\usepackage{hyperref}
\usepackage[capitalise,noabbrev]{cleveref}
\crefname{thm}{Theorem}{Theorems}
\crefname{lem}{Lemma}{Lemmas}
\crefname{cor}{Corollary}{Corollaries}
\crefname{prop}{Proposition}{Propositions}
\crefname{qu}{Question}{Questions}
\crefname{rem}{Remark}{Remarks}
\crefname{fig}{Figure}{Figures}
\usepackage{float}
\usepackage{booktabs}
\usepackage{tabularx}
\usepackage{array}
\hypersetup{
    colorlinks=true,
    linkcolor=blue,
    citecolor=magenta,
    filecolor=magenta,
    urlcolor=cyan,
    linktoc=all
}
\usepackage{titling}

\newlength{\bibitemsep}
\newlength{\bibparskip}
\let\oldthebibliography\thebibliography
\renewcommand\thebibliography[1]{%
  \oldthebibliography{#1}%
  \setlength{\parskip}{\bibitemsep}%
  \setlength{\itemsep}{\bibparskip}%
}

\makeatletter
\renewcommand{\paragraph}{%
  \@startsection{paragraph}{4}%
  {\z@}{1.6ex \@plus 1ex \@minus .2ex}{-0.5em}%
  {\normalfont\normalsize\bfseries}%
}
\makeatother

\theoremstyle{plain}
\newtheorem{thm}{Theorem}[section]

\newtheorem{cla}[thm]{Claim}
\newtheorem{cor}[thm]{Corollary}
\newtheorem{prop}[thm]{Proposition}
\newtheorem{qu}[thm]{Question}

\theoremstyle{definition}

\newcommand{\cI}{\mathcal{I}}

\def\final{1}
\def\iflong{\iffalse}
\ifnum\final=0
\newcommand{\kristof}[1]{{\color{red}[{\tiny \textbf{Kristóf:}  #1}]\marginpar{\color{red}*}}}
\newcommand{\bence}[1]{{\color{magenta}[{\tiny \textbf{Bence:}  #1}]\marginpar{\color{magenta}*}}}
\else
\newcommand{\kristof}[1]{}
\newcommand{\bence}[1]{}
\fi

\title{A Note on the List Chromatic Number of Two Matroids}

\thanksmarkseries{alph}
\author{
Bence Garami\thanks{Department of Computer Science and Information Theory, Faculty of Electrical Engineering and Informatics, Budapest University of Technology and Economics, Budapest, Hungary. Email: \texttt{garamibence@edu.bme.hu}.}
}
\date{}

\begin{document}
\maketitle
\thispagestyle{empty}

\begin{abstract}
We study list coloring of common independent sets of two matroids. We construct a graphic matroid $M_1$ and a partition matroid $M_2$ with common chromatic number two and common list chromatic number three, showing that the two parameters need not be equal. This resolves a question raised by Kir\'aly, later stated as a conjecture by Aharoni, Berger, Guo, and Kotlar. We also show that if two strongly base-orderable matroids are each $2$-colorable, then their intersection is $2$-list-colorable.

\end{abstract}

\section{Introduction}
\label{sec:intro}

Let $M=(E,\cI)$ be a loopless matroid. A \textit{coloring} of $M$ is a map $c\colon E\to C$ such that every color class $c^{-1}(\gamma)$ is independent in $M$. The \textit{chromatic number} $\chi(M)$ is the minimum number of colors in such a coloring, or equivalently the minimum number of independent sets whose union is $E$. By the matroid partition theorem of Edmonds~\cite{edmonds1965minimum}, an optimal coloring can be found in polynomial time in the independence-oracle model; see also~\cite[Chapter~42]{schrijver2003combinatorial}.

Now let $M_1=(E,\cI_1)$ and $M_2=(E,\cI_2)$ be two matroids on the same ground set. A coloring of $(M_1,M_2)$ is a coloring of $E$ in which every color class is independent in both matroids. We call the minimum number of colors in such a coloring the \textit{common chromatic number} and denote it by $\chi(M_1,M_2)$. In contrast to the single-matroid case, even deciding whether two colors suffice is hard in general; see~\cite{berczi2021complexity}.

List coloring asks whether a coloring can be chosen when the available colors depend on the element. A \textit{list assignment} is a family $L=\{L(e)\colon e\in E\}$ of finite sets of colors. An \textit{$L$-coloring} of a matroid is a coloring $c$ satisfying $c(e)\in L(e)$ for every $e\in E$. The \textit{list chromatic number} $\chi_\ell(M)$ is the least integer $k$ such that an $L$-coloring exists for every list assignment with $|L(e)|=k$ for all $e$. Seymour proved that $\chi_\ell(M)=\chi(M)$ for every loopless matroid~\cite{seymour1998note}.

The same definition applies to a pair of matroids: an $L$-coloring of $(M_1,M_2)$ assigns to each element a color from its list so that every monochromatic set is independent in both matroids. We call the corresponding parameter the \textit{common list chromatic number} and denote it by $\chi_\ell(M_1,M_2)$. Galvin's theorem on the list chromatic index of bipartite multigraphs gives an important special case~\cite{galvin1995list}. Indeed, the matchings of a bipartite multigraph are precisely the common independent sets of the two unit-capacity partition matroids defined by its two vertex classes. Hence $\chi_\ell(M_1,M_2)=\chi(M_1,M_2)$ whenever both matroids are unit-capacity partition matroids.

Kir\'aly and Pap studied the corresponding \textit{joint list-coloring property}~\cite{kiraly2010list}. Kir\'aly later asked whether $\chi_\ell(M_1,M_2)=\chi(M_1,M_2)$ for every pair of matroids~\cite{kiraly2013open}; the same question was considered by B\'erczi, Schwarcz, and Yamaguchi~\cite{berczi2021list} and was later restated as a conjecture by Aharoni, Berger, Guo, and Kotlar~\cite[Conjecture~5.4]{aharoni2025coloring}. For three matroids the equality was already known to fail: Kir\'aly gave an example of three partition matroids on six elements with common chromatic number two that are not list colorable from arbitrary two-element lists; see~\cite{berczi2021list}.

\subsection{Previous work}

For a single matroid, Seymour's theorem~\cite{seymour1998note} was strengthened by Laso\'n~\cite{lason2015list}, who allowed the prescribed list size to depend on the element, and by Laso\'n and Lubawski~\cite{lason2017online}, who proved the corresponding on-line statement.

For intersections, Kir\'aly and Pap~\cite{kiraly2010list} proved the joint list-coloring property for two transversal matroids, for two rank-two matroids, and in an arborescence setting. Aharoni and Berger~\cite{aharoni2006intersection} proved an ordinary $(p+q)$-coloring bound when the two matroids are respectively $p$- and $q$-colorable and $p$ divides $q$; Berger and Guo~\cite{berger2025coloring} removed the divisibility assumption and proved the stronger $(p+q)$-list-coloring statement. B\'erczi, Schwarcz, and Yamaguchi~\cite{berczi2021list} developed reductions to partition matroids and obtained list-coloring bounds for transversal, graphic, paving, and gammoid matroids. Aharoni, Berger, Guo, and Kotlar~\cite{aharoni2025coloring} proved in particular that $\chi_\ell(M_1,M_2)\le2\chi(M_1,M_2)$ for every pair of matroids.

\subsection{Our results}

Our main result shows that the common chromatic and common list chromatic numbers of two matroids need not be equal.

\begin{thm}\label{thm:counterintro}
There exist a graphic matroid $M_1$ and a partition matroid $M_2$ on the same ground set such that $\chi(M_1,M_2)=2$ and $\chi_\ell(M_1,M_2)=3$.
\end{thm}

This gives a counterexample to the conjecture mentioned above. The construction is rather simple: it consists of three copies of $K_4$ in the graphic matroid and a partition of their $18$ edges into pairs in the second matroid. The non-$2$-list-colorability has a direct proof; the upper bound $\chi_\ell(M_1,M_2)\le3$ follows from the partition-matroid reduction for graphic matroids in~\cite{berczi2021list} and Galvin's theorem.

We also show that if $M_1$ and $M_2$ are strongly base-orderable and $\chi(M_1),\chi(M_2)\le2$, then $\chi_\ell(M_1,M_2)\le2$. By a theorem of Davies and McDiarmid, the hypothesis is equivalent in this class to $\chi(M_1,M_2)\le2$.

\subsection{Preliminaries}
\label{sec:preliminaries}

All matroids in the paper are finite. We use standard matroid notation; see for example~\cite{schrijver2003combinatorial}. A \textit{ unit-capacity partition matroid} is specified by a partition $E=P_1\cup\dots\cup P_m$; its independent sets are the subsets meeting each $P_i$ in at most one element.

If $M=(E,\cI)$ and $N=(E,\mathcal{J})$ are matroids on the same ground set with $\mathcal{J}\subseteq\cI$, we say that $N$ is a \textit{reduction} of $M$. Thus every independent set of $N$ is independent in $M$. Reductions to partition matroids were used in~\cite{berczi2021list} to obtain bounds on list coloring of matroid intersections.

We also recall strong base orderability. A matroid $M$ is \textit{strongly base-orderable} if for every pair of bases $B_1,B_2$ there is a bijection $\varphi\colon B_1\to B_2$ such that $(B_1\setminus X)\cup\varphi(X)$ is a basis for every $X\subseteq B_1$. Strongly base-orderable matroids include all transversal matroids, and hence all partition matroids, and the class is closed under taking minors; see~\cite[Section~42.6c]{schrijver2003combinatorial}. Davies and McDiarmid proved that if $M_1$ and $M_2$ are strongly base-orderable, then $\chi(M_1,M_2)=\max\{\chi(M_1),\chi(M_2)\}$~\cite{davies1976disjoint}.

\section{List coloring of two matroids}
\label{sec:intersection}

\subsection{Separation of chromatic and list chromatic numbers}
\label{sec:counterexample}

\begin{proof}[Proof of \cref{thm:counterintro}]
Let $M_1$ be the direct sum of three copies of the cycle matroid of $K_4$. Let $E_i$ denote the edge set of the $i$th copy. In each copy, the six edges are partitioned into the three pairs of opposite edges, which we denote by $\{a_i^+,a_i^-\}$, $\{b_i^+,b_i^-\}$, and $\{c_i^+,c_i^-\}$ for $i=1,2,3$. We use the following elementary fact about $K_4$.

\begin{cla}\label{cla:k4}
Two complementary $3$-edge sets of $K_4$ are both spanning trees if and only if exactly one pair of opposite edges is split between them.
\end{cla}

\begin{proof}
If an opposite pair is contained in one part, then cardinality forces another opposite pair to be contained in the other part and the third pair to be split. Both parts are then connected and have three edges, hence are spanning trees. Conversely, suppose that no opposite pair is contained in either part. Then all three opposite pairs are split. A set obtained by choosing one edge from each opposite pair of $K_4$ is either a triangle or a star, and its complement is respectively a star or a triangle. Thus the two sets cannot both be spanning trees.
\end{proof}

Let $M_2$ be the unit-capacity partition matroid whose classes are
\[
\{a_1^+,a_1^-\},\ \{a_2^+,a_2^-\},\ \{a_3^+,a_3^-\},\ 
\{b_1^+,c_2^-\},\ \{b_2^+,c_3^-\},\ \{b_3^+,c_1^-\},\ 
\{c_1^+,b_2^-\},\ \{c_2^+,b_3^-\},\ \{c_3^+,b_1^-\}.
\]

\begin{figure}[t]
\centering
\begin{tikzpicture}[
    vertex/.style={circle,fill=black,inner sep=0pt,minimum size=4.8pt},
    aedge/.style={line width=1.15pt},
    bedge/.style={line width=1.15pt,dashed},
    cedge/.style={line width=1.15pt,densely dotted},
    edge label/.style={font=\small,fill=white,inner sep=1pt}
]
\foreach \i/\x/\lab in {1/0/{\{1,2\}},2/4.6/{\{2,3\}},3/9.2/{\{1,3\}}}{
    \coordinate (u\i) at (\x,1.9);
    \coordinate (v\i) at (\x+2.4,1.9);
    \coordinate (w\i) at (\x,0);
    \coordinate (z\i) at (\x+2.4,0);

    \draw[aedge] (u\i) -- node[edge label,above] {$a_\i^+$} (v\i);
    \draw[aedge] (w\i) -- node[edge label,below] {$a_\i^-$} (z\i);

    \draw[bedge] (u\i) -- node[edge label,left] {$b_\i^+$} (w\i);
    \draw[bedge] (v\i) -- node[edge label,right] {$b_\i^-$} (z\i);

    \draw[cedge] (u\i) -- (z\i);
    \draw[cedge] (v\i) -- (w\i);

    \node[edge label] at (\x+0.65,1.05) {$c_\i^+$};
    \node[edge label] at (\x+1.75,1.05) {$c_\i^-$};

    \node[vertex] at (u\i) {};
    \node[vertex] at (v\i) {};
    \node[vertex] at (w\i) {};
    \node[vertex] at (z\i) {};

    \node[font=\small] at (\x+1.2,-0.75) {$L(e)=\lab$ for $e\in E_\i$};
}
\end{tikzpicture}
\caption{The three copies $E_1,E_2,E_3$ of $K_4$ defining $M_1$. Solid, dashed, and dotted edges form the $a$-, $b$-, and $c$-pairs, respectively.}
\label{fig:construction}
\end{figure}
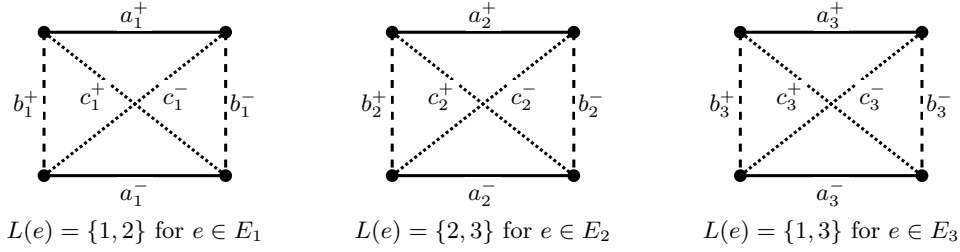

We first show that $\chi(M_1,M_2)=2$. Let
\[
X=\bigcup_{i=1}^3\{a_i^+,b_i^+,b_i^-\}\quad\text{and}\quad Y=\bigcup_{i=1}^3\{a_i^-,c_i^+,c_i^-\}.
\]
In each copy of $K_4$, the three edges of $X$ form a spanning tree, and the same is true for $Y$. Hence $X$ and $Y$ are independent in $M_1$. Every class of $M_2$ has one element in $X$ and one in $Y$, so both sets are independent in $M_2$ as well. Thus $\chi(M_1,M_2)\le2$. Since the entire ground set is dependent in $M_1$, equality holds.

We next give a two-element list assignment $L$ for which no $L$-coloring of $(M_1,M_2)$ exists. Give every edge in the first, second, and third copies of $K_4$ the lists $\{1,2\}$, $\{2,3\}$, and $\{1,3\}$, respectively; see \cref{fig:construction}. Suppose that an $L$-coloring exists. In each copy of $K_4$, only two colors are available. Each color class is a forest and therefore contains at most three edges. Since all six edges are colored, each of the two color classes has exactly three edges and hence is a spanning tree of that copy of $K_4$. By \cref{cla:k4}, exactly one of the three opposite pairs is bichromatic.

The pair $\{a_i^+,a_i^-\}$ is a class of $M_2$, so its two edges receive different colors. It is therefore the unique bichromatic opposite pair in the $i$th copy of $K_4$. Consequently, the $b_i$-pair and the $c_i$-pair are both monochromatic and receive different colors.

In the first copy, the $b_1$-pair has color $1$ or $2$. Suppose first that it has color $1$. Then the $c_1$-pair has color $2$. Since $\{c_3^+,b_1^-\}$ is a class of $M_2$, the $c_3$-pair must have color $3$. The $b_2$-pair cannot have color $2$, since $\{c_1^+,b_2^-\}$ is a class of $M_2$, and it cannot have color $3$, since $\{b_2^+,c_3^-\}$ is a class of $M_2$. This is impossible, as the only available colors in the second copy are $2$ and $3$.

Suppose instead that the $b_1$-pair has color $2$. Then the $c_1$-pair has color $1$. Since $\{b_3^+,c_1^-\}$ is a class of $M_2$, the $b_3$-pair must have color $3$. The $c_2$-pair cannot have color $2$, since $\{b_1^+,c_2^-\}$ is a class of $M_2$, and it cannot have color $3$, since $\{c_2^+,b_3^-\}$ is a class of $M_2$. This is again impossible. Hence $\chi_\ell(M_1,M_2)\ge3$.

The above argument already shows that $\chi_\ell(M_1,M_2)>\chi(M_1,M_2)$. For completeness, we show that the list chromatic number is exactly $3$. The graphic matroid $M_1$ is $2$-colorable. By the partition-matroid reduction theorem of B\'erczi, Schwarcz, and Yamaguchi~\cite[Theorem~1.6]{berczi2021list}, there is a unit-capacity partition matroid $N_1$ on the same ground set such that every $N_1$-independent set is $M_1$-independent and $\chi(N_1)\le2\chi(M_1)-1=3$. The matroid $M_2$ is already a unit-capacity partition matroid and is $2$-colorable. The common independent sets of $N_1$ and $M_2$ are the matchings of a bipartite multigraph whose two vertex classes correspond to the partition classes of $N_1$ and $M_2$. Since every class of $N_1$ has size at most $3$ and every class of $M_2$ has size at most $2$, this multigraph has maximum degree at most $3$. By Galvin's theorem~\cite{galvin1995list}, it has a proper list edge-coloring from arbitrary lists of size $3$. Each color class is independent in both $N_1$ and $M_2$, and hence also in both $M_1$ and $M_2$. Therefore $\chi_\ell(M_1,M_2)\le3$.
\end{proof}

\subsection{Two strongly base-orderable matroids}
\label{sec:sbo}

The example above does not settle the strongly base-orderable case. Indeed, $M_1$ has $M(K_4)$ as a minor, while $M(K_4)$ is not base-orderable~\cite{bonin2016infinite}; since strongly base-orderable matroids are minor-closed and in particular base-orderable, $M_1$ is not strongly base-orderable. When two colors suffice, the partition-matroid reduction viewpoint gives a positive result.

\begin{prop}\label{prop:reduction}
Let $M=(E,\cI)$ be a strongly base-orderable matroid, and suppose that $E=I\cup J$ for two disjoint independent sets $I$ and $J$. Then $M$ has a reduction to a unit-capacity partition matroid whose classes have size at most two.
\end{prop}

\begin{proof}
Extend $I$ to a basis $B_1=I\cup C$ with $C\subseteq J$, and extend $J$ to a basis $B_2=J\cup A$ with $A\subseteq I$. Then $B_1\cap B_2=A\cup C$, $B_1\setminus B_2=I\setminus A$, and $B_2\setminus B_1=J\setminus C$.

Apply strong base orderability in the contraction $M/(A\cup C)$ to the disjoint bases $I\setminus A$ and $J\setminus C$. There is a bijection $\varphi\colon I\setminus A\to J\setminus C$ such that $((I\setminus A)\setminus X)\cup\varphi(X)$ is a basis of $M/(A\cup C)$ for every $X\subseteq I\setminus A$. Let $N$ be the unit-capacity partition matroid whose classes are the pairs $\{x,\varphi(x)\}$ for $x\in I\setminus A$, together with the singletons $\{e\}$ for $e\in A\cup C$.

To see that $N$ is a reduction of $M$, let $T$ be independent in $N$, and let $X$ be the set of those $x\in I\setminus A$ for which $\varphi(x)\in T$. Then $T\subseteq ((I\setminus A)\setminus X)\cup\varphi(X)\cup A\cup C$. The first two terms form a basis of $M/(A\cup C)$, while $A\cup C$ is independent in $M$, so the set on the right is a basis of $M$. Hence $T$ is independent in $M$.
\end{proof}

\begin{cor}\label{cor:sbo}
Let $M_1$ and $M_2$ be strongly base-orderable matroids on a common ground set. If $\chi(M_1)\le2$ and $\chi(M_2)\le2$, then $\chi_\ell(M_1,M_2)\le2$. In particular, if $\chi(M_1,M_2)=2$, then $\chi_\ell(M_1,M_2)=2$.
\end{cor}

\begin{proof}
Applying \cref{prop:reduction} to each matroid gives unit-capacity partition matroids $N_1$ and $N_2$ such that every $N_j$-independent set is $M_j$-independent and every partition class of $N_j$ has size at most two. The common independent sets of $N_1$ and $N_2$ are the matchings of a bipartite multigraph of maximum degree at most two. By Galvin's theorem~\cite{galvin1995list}, this multigraph has a proper list edge-coloring from arbitrary lists of size two. The resulting color classes are independent in both $M_1$ and $M_2$, so $\chi_\ell(M_1,M_2)\le2$. The equivalent formulation follows from the theorem of Davies and McDiarmid mentioned in \cref{sec:preliminaries}.
\end{proof}

\section{Concluding questions}
\label{sec:open}

The equality $\chi_\ell=\chi$ fails for two matroids, so it is natural to ask how large the additive gap can be. Aharoni, Berger, Guo, and Kotlar~\cite{aharoni2025coloring} proved $\chi_\ell(M_1,M_2)\le2\chi(M_1,M_2)$ for every pair of matroids, while \cref{thm:counterintro} shows that an additive gap of one is possible.

\begin{qu}\label{qu:gap}
Is there a constant $C$ such that $\chi_\ell(M_1,M_2)\le \chi(M_1,M_2)+C$ for every pair of matroids $M_1,M_2$? In particular, does $\chi_\ell(M_1,M_2)\le \chi(M_1,M_2)+1$ always hold?
\end{qu}

When one matroid is graphic and the other is a partition matroid, the reduction theorem of B\'erczi, Schwarcz, and Yamaguchi~\cite[Theorem~1.6]{berczi2021list} and Galvin's theorem give
\[
\chi_\ell(M_1,M_2)\le \max\{2\chi(M_1)-1,\chi(M_2)\}\le2\chi(M_1,M_2)-1.
\]
For $\chi(M_1,M_2)=2$, this bound equals $3$ and is attained by the example in \cref{thm:counterintro}. Thus the bound is tight when the common chromatic number is two; for $\chi(M_1,M_2)\ge3$, it remains open whether it can be improved.

The strongly base-orderable setting remains open beyond two colors. Davies and McDiarmid's theorem determines the common chromatic number in this class~\cite{davies1976disjoint}, and \cref{cor:sbo} gives the list analogue when that number is two. Kir\'aly and Pap~\cite{kiraly2010list} proved equality for the subclass of transversal matroids.

\begin{qu}\label{qu:sbo}
Does $\chi_\ell(M_1,M_2)=\chi(M_1,M_2)$ hold whenever $M_1$ and $M_2$ are strongly base-orderable?
\end{qu}

\bigskip

\paragraph{Acknowledgements.}
The author is grateful to Kristóf Bérczi and Tamás Fleiner for helpful discussions.

\paragraph{Use of generative AI.}
No AI assistance was used in obtaining the results of this paper.

\bibliographystyle{abbrv}
\bibliography{list_coloring}

\end{document}